\documentclass[a4paper,11pt,reqno]{amsart}
\usepackage{amsfonts,amsmath,amsthm,amssymb,mathrsfs}
\usepackage{color}
\usepackage{amsmath,amssymb,amsfonts,bm}
\usepackage{graphicx}
\usepackage{CJKutf8}
\usepackage{newunicodechar}
\usepackage{xcolor}
\numberwithin{equation}{section}
\usepackage[pdfstartview=FitH,colorlinks=true]{hyperref}
\hypersetup{urlcolor=red, linkcolor=blue,citecolor=red, CJKbookmarks=true}
\allowdisplaybreaks

\newtheorem{theorem}{Theorem}[section]

\newtheorem{lemma}[theorem]{Lemma}

\theoremstyle{definition}
\newtheorem{definition}[theorem]{Definition}
\newtheorem{remark}{Remark}[section]

\newcommand{\abs}[1]{\left\vert#1\right\vert}

\newcommand{\norm}[1]{\left\Vert#1\right\Vert}

\newcommand{\inner}[1]{\left(#1\right)}

\makeatletter

\newcommand{\Rmnum}[1]{\expandafter\@slowromancap\romannumeral #1@}
\makeatother

\DeclareMathOperator\divg{div}

\title[The Boussineq equations]
{Analytical smoothing effect of solution\\
for the Boussinesq equations}
\author{F. Cheng \& C.-J. Xu}
\date{}

\address{\noindent \textsc{Feng Cheng, School of Mathematics and Statistics, Wuhan university 430072, Wuhan, P.R. China}}
\email{chengfengwhu@whu.edu.cn}

\address{\noindent \textsc{Chao-Jiang Xu, School of Mathematics and Statistics, Wuhan university 430072, Wuhan, P.R. China and 
 Universit\'e de Rouen, CNRS UMR 6085, Laboratoire de Math\'ematiques, 76801 Saint-Etienne du Rouvray, France}}
\email{Chao-Jiang.Xu@univ-rouen.fr}
\begin{document}


\begin{abstract}
In this paper, we study the analytical smoothing effect of Cauchy problem for the incompressible Boussinesq equations. Precisely, we use the Fourier method to prove that the Sobolev $H^1$-solution  to the incompressible Boussinesq equations in periodic domain is analytic for any positive time. So the incompressible Boussinesq equation admet exactly same smoothing effect properties of incompressible Navier-Stokes equations.
\end{abstract}

\maketitle

\section{Introduction}

In this paper, we consider the following incompressible Boussinesq  equations on the torus $\mathbb{T}^N=[0,2\pi]^N$ with $N=2$ or $3$,
\begin{equation}\label{1.1}
\left\{
\begin{aligned}
 &\frac{\partial u}{\partial t}-\nu\Delta u+u\cdot\nabla u+\nabla p=\theta e_N, \\
 &\frac{\partial \theta}{\partial t}-\kappa\Delta \theta+u\cdot\nabla \theta=0,\\
 &\nabla\cdot u=0,\\
 &u(0,x)=u_0(x),\quad \theta(0,x)=\theta_0(x),
\end{aligned}
\right.
\end{equation}
where $u(t,x)=(u_1,\ldots,u_N)(t,x), (t,x)\in \mathbb{R}^+\times\mathbb{T}^N$, is the velocity vector field, $p=p(t,x)$ is the scalar pressure, $\theta=\theta(t,x)$ is the scalar temperature in the content of thermal convection and the density in the modeling of geophysical fluids, $\nu>0$ is the viscosity, and $\kappa>0$ is the thermal diffusivity, and $e_N=(0,\ldots, 1)$ is the unit vector in the $x_N$-direction. In addition to \eqref{1.1}, we assume that $u,\theta,p$ are periodic for the spatial variable and the average value of $u,\theta,p$ on ${\mathbb{T}^N}$ vanishes:
\begin{equation}\label{1.2}
 \int_{\mathbb{T}^N} u(t,x)dx = 0,\,\,\, \int_{\mathbb{T}^N} \theta(t,x)dx=\int_{\mathbb{T}^N} p(t,x)dx=0,\quad \forall\, t\geq0.
\end{equation}
In physics, the Boussinesq system \eqref{1.1} is commonly used to model large scale atmospheric and oceanic flows, for example, tornadoes, cyclones, and hurricanes. It describes the dynamics of fluid influenced by gravitational force, which plays an very important role in the study of Raleigh-Bernard convection, see \cite{G,MB,MB1,P}.

In mathematics, the Boussinesq system is one of the most commonly used simplifed models to understand some key features of the 3-D incompressible Navier Stokes equations and Euler equations. Actually, if we set $\theta\equiv 0$, the Boussinesq system \eqref{1.1} reduces to the incompressible Navier Stokes equations. 

The well-posedness of the Boussinesq system \eqref{1.1} has been studied extensively in recent years. The global well-posedness of weak solutions, or strong solutions in the case of small data for the  Boussinesq equations has been considered by many authors. See, e.g. \cite{AH,BS,Chae3,DP1,DP2,DG,ST}. %\textcolor{red}
The global existence and uniqueness of smooth solutions to the 2D Boussinesq system  with partial viscosity has also been studied, see \cite{Chae1,Chae2,W,HL}. However, similar to three-dimensional Navier Stokes equation, the global existence or finite time blow-up of smooth solutions for the 3-D incompressible Boussinesq equations is still an open problem.

In this paper we study the analytical regularity of the strong solution to the incompressible Boussinesq equation \eqref{1.1}, which is similar to the analytical regularity of Navier Stokes equation. Since Foias and Temam \cite{FT} studied the Gevrey class regularity of Navier Stokes equations in 1989, there are many works on the Gevrey class regularity of solutions for many kinds of equations, see \cite{B,PV}. In two dimensions, the authors of \cite{BG} studied Gevrey class regularity for the two-dimensional Newton-Boussinesq equations. They considered the vorticity instead of the velocity to eliminate the pressure by transforming the velocity equation into the vorticity equation. In this paper, we use the Leray projection operator to deal with the pressure and the method can be applied to higher dimensions, which improves the previous work. 

The paper is organized as follows. In Section \ref{Section 2}, we will give some notations and state our main results. In Section \ref{Section 3}, we first recall some known results and then give some lemmas which are needed to prove the Theorem \ref{Theorem 2.1}. In Section \ref{Section 4}, we give the proof of the main Theorem \ref{Theorem 2.1}. 

\section{Notations and Main Theorems}\label{Section 2}

In this section, we will give some notations and function spaces which will be used throughout the following arguments. Throughout the paper, $C$ denotes a generic constant which may vary from line to line.

Denote by $C_p^\infty({\mathbb{T}^N})^N$ the $\mathbb{C}^N$-valued vector functions space of smooth functions satisfying periodic boundary condition, i.e., for $\phi\in C^\infty(\mathbb{R}^N)$,
$$
 \phi(x_1, \cdots, x_N)=\phi(x_1+2n\pi,\ldots,x_N+2n\pi),\ \forall\ n\in\mathbb{Z}\, .
$$
Denote
$$
 \mathcal{V}=\bigg\{\varphi\in C_p^\infty(\mathbb{T}^N)^N;\,\, \nabla\cdot \varphi=0 \bigg\}.
$$
We denote $L^2(\mathbb{T}^N)^N$ the $\mathbb{C}^N$-valued vector functions space of each vector component in $L^2(\mathbb{T}^N)$, and it is a Hilbert space for the inner product,
$$
 \inner{f,\,\, g}_{L^2(\mathbb{T}^N)}=\sum_{k=1}^N\int_{\mathbb{T}^N}f_k(x)\, \bar{g}_k(x)dx,
$$
where $f=(f_1,\ldots,f_N), g=(g_1,\ldots,g_N)\in L^2(\mathbb{T}^N)^N.$
Let $m>0$ be an integer, we denote $H^m(\mathbb{T}^N)^N$ be the $\mathbb{C}^N$-valued vector function space of each vector component in Sobolev space $H^m(\mathbb{T}^N)$, and the corresponding norm $\norm{\,\cdot\,}_m$ is
$$
 \norm{f}_m^2=\sum_{\abs{\alpha}\leq m}\sum_{k=1}^N\norm{\partial^\alpha f_k}_{L^2(\mathbb{T}^N)}^2\, .
$$

We denote by $L^2_\sigma(\mathbb{T}^N)$ the completion of $\mathcal{V}$ under the norm $\norm{\, \cdot\, }_{L^2(\mathbb{T}^N)}$ and similarly $H^m_\sigma(\mathbb{T}^N)$ is the closure of $\mathcal{V}$ under the norm $\norm{\, \cdot\,}_{m}$.
Using Fourier series expansion, we can identify $L^2_\sigma(\mathbb{T}^N)$ with 
\begin{align*}
 L^2_\sigma(\mathbb{T}^N)=\bigg\{u=\sum_{j\in\mathbb{Z}^N}\hat{u}_j e^{ij\cdot x},&\quad \hat{u}_j\cdot j=0,\quad \hat{u}_{-j}=\overline{\hat{u}_j},\ \hat{u}_0=0,\\
 &\quad \norm{u}^2_{L^2(\mathbb{T}^N)}=(2\pi)^N\sum_{j\in\mathbb{Z}^N}|\hat{u}_j|^2<\infty \bigg\},
\end{align*}
where $\hat{u}_0=0$ meets the condition (1.2). Let $r>0$ be a real number, we can then identify the Sobolev space $H^r_\sigma(\mathbb{T}^N)$ as
\begin{align*}
 H^r_\sigma(\mathbb{T}^N)=\bigg\{u=\sum_{j\in\mathbb{Z}^N}\hat{u}_j e^{ij\cdot x},&\quad \hat{u}_j\cdot j=0,\quad \hat{u}_{-j}=\overline{\hat{u}_j},\ \hat{u}_0=0,\\
 &\quad \norm{u}^2_r=(2\pi)^N\sum_{j\in\mathbb{Z}^N}\abs{j}^{2r}|\hat{u}_j|^2<\infty \bigg\},
\end{align*}
Recall that we say a smooth function $f$ belongs to Gevery class $G^s (\mathbb{T}^N)$  for some $s>0$, if there exists $C, \tau>0$ such that 
$$
\|\partial^\alpha f\|_{L^2(\mathbb{T}^N)} \leq C\frac{|\alpha|!^s}{\tau^{|\alpha|}},\quad \forall\, \alpha\in\mathbb{N}^N\, .
$$
The parameter $\tau$ is called the radius of Gevrey class $s$. In particular, when $s=1$, it  is the analytical functions.
Let us define the Zygmund operator $\Lambda=\sqrt{-\Delta}$ with
\begin{equation*}
 \Lambda u=\sum_{j\in\mathbb{Z}^N}\abs{j}\hat u_j e^{i j\cdot x},\quad u
 \in H^1_\sigma(\mathbb{T}^N).
\end{equation*}
The same definition of $ \Lambda \theta$ for the scalar function 
$\theta\in H^1(\mathbb{T}^N)$. 

We define now the function space $\mathcal{D}(\Lambda^r e^{\tau \Lambda^{1/s}})$ for $r, s, \tau>0$,  $u\in \mathcal{D}(\Lambda^r e^{\tau \Lambda^{1/s}})$, if $u\in L^2_\sigma(\mathbb{T}^N)$ and
$$
 \|\Lambda^r e^{\tau\Lambda^{1/s}}u\|_{L^2(\mathbb{T}^N)}^2=(2\pi)^N \sum_{j\in\mathbb{Z}^N}\abs{j}^{2r}e^{2\tau|j|^{1/s}}|\hat{u}_j|^2<\infty.
$$
For the scalar function $\theta$, we also write $\theta\in \mathcal{D}(\Lambda^r e^{\tau \Lambda^{1/s}})$, which means $\theta\in L^2((\mathbb{T}^N)$,   $\hat \theta_0=0$ and
$$
 \|\Lambda^r e^{\tau\Lambda^{1/s}} \theta\|_{L^2(\mathbb{T}^N)}^2=(2\pi)^N \sum_{m\in\mathbb{Z}^N}\abs{m}^{2r}e^{2\tau|m|^{1/s}}|\hat{\theta}_m|^2<\infty.
$$
It was proved in \cite{LO} that  $\mathcal{D}(\Lambda^r e^{\tau \Lambda^{1/s}})\subset 
G^s (\mathbb{T}^N)$ .

With these preparations, we can state our main results.

\begin{theorem}\label{Theorem 2.1}
If the initial data $(u_0,\theta_0)\in H^1_\sigma(\mathbb{T}^N)\times H^1(\mathbb{T}^N)$. Then the following results hold:\\
(1).For two dimensions case, $N=2$, the Boussinesq equations \eqref{1.1} admit an  unique global solution $(u,\, \theta)$ satisfying
$$
 u(t,x)\in C\big([0,\, +\infty[;\, \mathcal{D}(\Lambda e^{t\Lambda})\big),\quad \theta(t,x)\in C\big([0,\, +\infty[;\, \mathcal{D}(\Lambda e^{t\Lambda})\big)\, .
$$
(2).For three dimensions case, $N=3$, the Boussinesq equations \eqref{1.1} admit a unique local solution $(u,\theta)$ satisfying
$$
 u(t,x)\in C\big([0,T_1];\mathcal{D}(\Lambda e^{t\Lambda}) \big),\quad \theta(t,x)\in C\big([0,T_1];\mathcal{D}(\Lambda e^{t\Lambda})\big),
$$
where $T_1>0$ depends on the initial data $(u_0,\theta_0)$.
\end{theorem}
\begin{remark}.

\noindent
{\bf 1.} Following the arguments of Foias and Temam \cite{FT}, one can improve the regularity of time $t$ for the solution $(u,\theta)$ by extending to the complex plane and showing that the solution is actually analytic in time variable. Since the computations are standard, we omit these details.

\noindent
{\bf 2.} The analytical smoothing effect of Theorem \ref{Theorem 2.1} imply the Gevery 
smoothing effect  in $\mathcal{D}(\Lambda^r e^{\tau \Lambda^{1/s}})\subset 
G^s (\mathbb{T}^N)$ for any $s\ge 1$. 
\end{remark}

\section{Premilinary results}\label{Section 3}
In this Section, we recall some known results on  the incompressible Boussinesq system \eqref{1.1}.
Let us first recall the definition of weak solutions and strong solutions for the Boussinesq system \eqref{1.1}.
\begin{definition}[weak solutions]\label{Def 3.1}
For any $T>0$, we call $(u,\theta)$ the weak solution of the Boussinesq system \eqref{1.1}, if
\begin{align*}
 u\in C\big([0,T]\,;L^2_\sigma (\mathbb{T}^N)\big)\cap L^2\big([0,T]\,; H^1_\sigma(\mathbb{T}^N) \big),\\
  \theta\in C\big([0,T]\,;L^2 (\mathbb{T}^N)\big)\cap L^2\big([0,T]\,;H^1(\mathbb{T}^N) \big),
\end{align*}
and $u$ satisfies the following weak formulation
\begin{equation*}%\label{3.1}
 \int_0^t \int_{\mathbb{T}^N} \partial_t\varphi\cdot u+\nabla\varphi:u\otimes u+\theta e_N\cdot \varphi+\nu \Delta\varphi\cdot u dxds=0,
\end{equation*}
for all $\varphi\in C^\infty_0\big([0,T]\times {\mathbb{T}^N}\,;\mathbb{R}^N\big)$ with $\nabla\cdot\varphi=0$,
and $\theta$ satisfies the following weak formulation
\begin{equation*}%\label{3.2}
 \int_0^t \int_{{\mathbb{T}^N}} \theta\partial_t \psi +\theta (u\cdot\nabla\psi)  +\kappa\theta \Delta\psi =0,
\end{equation*}
for all $\psi\in C^\infty_0\big([0,T]\times{\mathbb{T}^N}\,;\,\,\mathbb{R}\big)$.
\end{definition}
The pressure term $p$ in the velocity equation of \eqref{1.1} is the Lagrange multiplier, which is eliminated by taking the $L^2$-inner product with the divergence-free vector field. However, it can be determined by
\begin{equation*}
\Delta p=-\divg(u\cdot\nabla u)+\divg(\theta e_N),
\end{equation*}
with $p$ satisfying periodic boundary conditions and \eqref{1.2}.

Analogous to the incompressible Navier Stokes equation, the global existence of weak solutions to the Boussinesq equations is standard, see \cite{CD, DG, HK1, HK2} and references therein.

\begin{theorem}[weak solutions]\label{Theorem 3.2}
Let $(u_0, \theta_0)\in L^2_\sigma(\mathbb{T}^N)\times L^2(\mathbb{T}^N)$. There exists a weak solution $(u,\theta)$ of the Boussinesq equations \eqref{1.1} such that,
\begin{align*}
 u\in C\big([0,T]\,;L^2_\sigma (\mathbb{T}^N)\big)\cap L^2\big([0,T]\,; H^1_\sigma(\mathbb{T}^N) \big),\\
  \theta\in C\big([0,T]\,;L^2 (\mathbb{T}^N)\big)\cap L^2\big([0,T]\,;H^1(\mathbb{T}^N) \big)\,.
\end{align*}
Such solutions satisfies, for all $t\in [0,T]$, the energy inequalities
\begin{equation*}%\label{3.3}
 \norm{\theta(t)}_{L^2(\mathbb{T}^N)}^2+2\kappa\int_0^t\norm{\nabla \theta(s)}_{L^2(\mathbb{T}^N)}^2 ds \leq \norm{\theta_0}_{L^2(\mathbb{T}^N)}^2,
\end{equation*}
and
\begin{align*}%\label{3.4}
 \norm{u(t)}_{L^2(\mathbb{T}^N)}^2 &+2\nu\int_0^t \norm{\nabla u(s)}_{L^2(\mathbb{T}^N)}^2 ds
     \leq \norm{u_0}_{L^2(\mathbb{T}^N)}^2+Ct^2\norm{\theta_0}_{L^2(\mathbb{T}^N)}^2,
\end{align*}
for all $T>0$, and some constant $C>0$.
\end{theorem}
\begin{remark}
For two dimensions case, $N=2$, the weak solution obtained in Theorem \ref{Theorem 3.2} is unique. Whether the uniqueness of weak solution for three dimensions case $N=3$ is still unknown.
\end{remark}
In the following, we study a class of more regular solutions to the Boussinesq equations, which is called strong solutions analogous to the Navier-Stokes equations and introduced by Temam in \cite{RT}.
\begin{definition}[strong solutions]
Let $(u_0,\theta_0)\in H^1_\sigma(\mathbb{T}^N)\times H^1(\mathbb{T}^N)$, 
Then $(u,\theta)$ is a strong solution of the Boussinesq equations \eqref{1.1} if
\begin{align*}
 u\in C\big([0,T]\,; H^1_\sigma (\mathbb{T}^N)\big)\cap L^2\big([0,T]\,; H^2_\sigma(\mathbb{T}^N) \big),\\
  \theta\in C\big([0,T]\,; H^1 (\mathbb{T}^N)\big)\cap L^2\big([0,T]\,; H^2(\mathbb{T}^N) \big)\, ,
\end{align*}
and $(u, \theta)$ satisfies (3.1) and (3.2).
\end{definition}

The existence and uniqueness of strong solutions for the incompressible Boussinesq equations \eqref{1.1} is analougous to the results of Navier Stokes equations, which we state as follows.
\begin{theorem}\label{Theorem 3.4}
Let $(u_0,\theta_0)\in H^1_\sigma(\mathbb{T}^N)\times H^1(\mathbb{T}^N)$. Then the following results hold. \\
(1). For two dimensions case $N=2$,  for any $T>0$, the Boussinesq equations \eqref{1.1} admit an unique strong solution $(u, \theta)$ on $[0,T]$ satisfying
\begin{align*}
 u\in C\big([0,T]\,; H^1_\sigma (\mathbb{T}^2)\big)\cap L^2\big([0,T]\,; H^2_\sigma(\mathbb{T}^2) \big),\\
  \theta\in C\big([0,T]\,; H^1 (\mathbb{T}^2)\big)\cap L^2\big([0,T]\,; H^2(\mathbb{T}^2) \big)\, .
\end{align*}
(2). For three dimensions $N=3$, there exists $T_\ast>0$ such that the Boussinesq equations \eqref{1.1} admit a unique strong solution $(u, \theta)$ on $[0,T_\ast]$ satisfying
\begin{align*}
 u\in C\big([0,T]\,; H^1_\sigma (\mathbb{T}^3)\big)\cap L^2\big([0,T]\,; H^2_\sigma(\mathbb{T}^3) \big),\\
  \theta\in C\big([0,T]\,; H^1 (\mathbb{T}^3)\big)\cap L^2\big([0,T]\,; H^2(\mathbb{T}^3) \big)\, ,
\end{align*}
where $T_\ast>0$ depends on the initial data $(u_0, \theta_0)$.
\end{theorem}
\begin{proof}
For the self-contenant of paper, we give the {\it \`a priori} estimate of  the solutions, then the construction of approximate solutions follows from the standard Galerkin method which we refer readers to \cite{RT}.

Let us assume that $(u,\theta)$ is the smooth solution of the incompressible Boussinesq equations. We multiply the velocity equation of \eqref{1.1} by $\Lambda^2 u$ and integrate over ${\mathbb{T}^N}$. After integrating by parts, we have
\begin{align}\label{3.6}
 \frac{1}{2}\frac{d}{dt} \|\Lambda u(t,\cdot)\|_{L^2(\mathbb{T}^N)}^2 &+\nu \|\Lambda^2 u\|_{L^2(\mathbb{T}^N)}^2\notag \\
& = \big(\Lambda(\theta e_N) ,\Lambda u \big)_{L^2(\mathbb{T}^N)} - \big(u\cdot\nabla u ,\Lambda^2 u \big)_{L^2(\mathbb{T}^N)},
\end{align}
where $\big(\nabla p,\Lambda^2 u\big)_{L^2(\mathbb{T}^N)}$ vanishes because $u$ is divergence-free and $\Lambda^2=-\Delta$ does not violate this property on the torus.
Then we multiply the thermal equation of \eqref{1.1} by $\Lambda^2\theta$ and integrate over ${\mathbb{T}^N}$. Integrating by parts, we obtain
\begin{align}\label{3.7}
 \frac{1}{2}\frac{d}{dt} \|\Lambda \theta(t,\cdot)\|_{L^2(\mathbb{T}^N)}^2 +\kappa \|\Lambda^2 \theta\|_{L^2(\mathbb{T}^N)}^2 =-\big( u\cdot\nabla \theta,\Lambda^2\theta\big)_{L^2(\mathbb{T}^N)}.
\end{align}
We first recall the following Sobolev inequality (see \cite{RT},  now usually called the Gagliardo-Nirenberg inequality),
\begin{align}
\|f\|_{L^\infty(\mathbb{T}^2)} &\leq C\|f\|_{L^2(\mathbb{T}^2)}^{1/2}\|\Lambda^2 f\|_{L^2(\mathbb{T}^2)}^{1/2},\quad\forall f\in H^2(\mathbb{T}^2), \label{GN-1}\\
 \|f\|_{L^\infty(\mathbb{T}^3)} &\leq C\|\Lambda f\|_{L^2(\mathbb{T}^3)}^{1/2}\|\Lambda^2 f\|_{L^2(\mathbb{T}^3)}^{1/2},\quad\forall f\in H^2(\mathbb{T}^3).\label{GN-2}
\end{align}
Using the Cauchy-Schwartz inequality and the Sobolev inequality, we have
\begin{align}
 \big|\big(u\cdot\nabla u ,\Lambda^2 u \big)_{L^2(\mathbb{T}^2)} \big| 
 &\leq \|u\cdot\nabla u\|_{L^2(\mathbb{T}^2)}\|\Lambda^2 u\|_{L^2(\mathbb{T}^2)}\notag\\
 &\leq \|u\|_{L^\infty(\mathbb{T}^2)}\|\nabla u\|_{L^2(\mathbb{T}^2)}\|\Lambda^2 u\|_{L^2(\mathbb{T}^2)}\notag\\
 &\leq C \|u\|_{L^2(\mathbb{T}^2)}^{1/2}\|\Lambda u\|_{L^2(\mathbb{T}^2)} \|\Lambda^2 u\|_{L^2(\mathbb{T}^2)}^{3/2}, \label{3.8}
 \end{align}
 and
 \begin{align}
  \big|\big(u\cdot\nabla u ,\Lambda^2 u \big)_{L^2(\mathbb{T}^3)} \big|
  &\leq \|u\cdot\nabla u\|_{L^2(\mathbb{T}^3)}\|\Lambda^2 u\|_{L^2(\mathbb{T}^3)}\notag\\
 &\leq \|u\|_{L^\infty(\mathbb{T}^3)}\|\nabla u\|_{L^2(\mathbb{T}^3)}\|\Lambda^2 u\|_{L^2(\mathbb{T}^3)}\notag\\  
  &\leq C \|\Lambda u\|_{L^2(\mathbb{T}^3)}^{3/2}\|\Lambda^2 u\|_{L^2(\mathbb{T}^3)}^{3/2}\, . \label{3.9}
\end{align}
We have also
\begin{align}
 \big|\big(u\cdot\nabla \theta ,\Lambda^2\theta \big)_{L^2(\mathbb{T}^2)} \big| 
 &\leq \|u\cdot\nabla\theta\|_{L^2(\mathbb{T}^2)}\|\Lambda^2 \theta\|_{L^2(\mathbb{T}^2)}\notag\\
 &\leq C \|u\|_{L^2(\mathbb{T}^2)}^{1/2}\|\Lambda^2u\|_{L^2(\mathbb{T}^2)}^{1/2}\notag\\
 &\qquad\qquad\qquad\times \|\nabla \theta\|_{L^2(\mathbb{T}^2)}\|\Lambda^2\theta\|_{L^2(\mathbb{T}^2)}, \label{3.10}
  \end{align}
 and
 \begin{align}
  \big|\big(u\cdot\nabla \theta ,\Lambda^2\theta \big)_{L^2(\mathbb{T}^3)} \big| 
 &\leq C\|\Lambda u\|_{L^2(\mathbb{T}^3)}^{1/2}\|\Lambda^2 u\|_{L^2(\mathbb{T}^3)}^{1/2}\|\Lambda\theta\|_{L^2}\|\Lambda^2\theta\|_{L^2(\mathbb{T}^3)}. \label{3.11}
\end{align}

\noindent
{\bf Case of  $N=2$.}  we add \eqref{3.6} and \eqref{3.7} and apply the Young's inequality on \eqref{3.8} and \eqref{3.10},
\begin{align*}
& \frac{1}{2}\frac{d}{dt}\big(\|\Lambda u\|_{L^2(\mathbb{T}^2)}^2 +\|\Lambda\theta\|_{L^2(\mathbb{T}^2)}^2 \big) +\nu\|\Lambda^2 u \|_{L^2(\mathbb{T}^2)}^2+\kappa\|\Lambda^2\theta\|_{L^2(\mathbb{T}^2)}^2\\
  &\leq C\|u\|_{L^2(\mathbb{T}^2)}^{1/2}\|\Lambda u\|_{L^2(\mathbb{T}^2)}\|\Lambda^2u\|_{L^2(\mathbb{T}^2)}^{3/2}+\|\Lambda(\theta e_n)\|_{L^2(\mathbb{T}^2)}\|\Lambda u\|_{L^2(\mathbb{T}^2)}\\
 &\qquad +C\|u\|_{L^2(\mathbb{T}^2)}^{1/2}\|\Lambda^2u\|_{L^2(\mathbb{T}^2)}^{1/2}\|\Lambda \theta\|_{L^2(\mathbb{T}^2)}\|\Lambda^2\theta\|_{L^2(\mathbb{T}^2)}\\
  &\leq \frac{\nu}{2}\|\Lambda^2u\|_{L^2(\mathbb{T}^2)}^2 +\frac{\kappa}{4}\|\Lambda^2\theta\|_{L^2(\mathbb{T}^2)}^2 +C_\nu \|u\|_{L^2}^2 \|\Lambda u\|_{L^2(\mathbb{T}^2)}^4\\ 
  &\qquad+C_{\kappa,\nu} \|u\|^2_{L^2(\mathbb{T}^2)}\|\Lambda\theta\|_{L^2(\mathbb{T}^2)}^4+\|\Lambda \theta\|_{L^2(\mathbb{T}^2)}\|\Lambda u\|_{L^2(\mathbb{T}^2)}.
\end{align*}
Thus
\begin{align}\label{3.12}
& \frac{d}{dt}\big(\|\Lambda u\|_{L^2(\mathbb{T}^2)}^2 +\|\Lambda\theta\|_{L^2(\mathbb{T}^2)}^2 \big) +\nu\|\Lambda^2 u \|_{L^2(\mathbb{T}^2)}^2+\kappa\|\Lambda^2\theta\|_{L^2(\mathbb{T}^2)}^2
 \notag\\
  &\leq C_{\kappa,\nu} \|u\|_{L^2}^2( \|\Lambda u\|_{L^2(\mathbb{T}^2)}^4+\|\Lambda\theta\|_{L^2(\mathbb{T}^2)}^4)+\|\Lambda \theta\|_{L^2(\mathbb{T}^2)}\|\Lambda u\|_{L^2(\mathbb{T}^2)}.
\end{align}
Thus if we denote 
$$
Y(t)=1+\|\Lambda u(t,\cdot)\|_{L^2(\mathbb{T}^2)}^2+\|\Lambda \theta(t,\cdot)\|_{L^2(\mathbb{T}^2)}^2,
$$ 
we have
\begin{align*}%\label{3.13}
 \frac{d}{dt}Y(t) \leq C_{\nu,\kappa}^\prime (1+\|u\|_{L^2(\mathbb{T}^2)}^2)\big(1+\|\Lambda u\|_{L^2(\mathbb{T}^2)}^2+\|\Lambda\theta\|_{L^2(\mathbb{T}^2)}^2 \big)Y(t),
\end{align*}
where $C_{\nu,\kappa}^\prime$ is some large constant depending on $\nu,\kappa$.

Then  the Gronwall's inequality imply,
\begin{align*}
 Y(t) &\leq Y(0)\exp\bigg[C_{\nu,\kappa}^\prime \int_0^t (1+\|u(s,\cdot)\|_{L^2}^2)\big(1+\|\Lambda u(s)\|_{L^2}^2+\|\Lambda\theta(s)\|_{L^2}^2 \big)ds \bigg]\,. 
\end{align*}

Noting that the strong solution is always weak solution, using Theorem \ref{Theorem 3.2}, we have
\begin{align*}
 u\in C\big([0,T]\,;L^2_\sigma\big)\cap L^2\big([0,T]\,;H^1_\sigma\big),\quad \theta\in C\big([0,T]\,;L^2\big)\cap L^2\big([0,T]\,;H^1\big),
\end{align*}
and
\begin{align*}
 &\sup_{t\in [0,T]}\|u(t,\cdot)\|_{L^2}\leq \|u_0\|_{L^2},\\
& \int_0^t \|\nabla u(s,\cdot)\|_{L^2}^2ds \leq \frac{\|u_0\|_{L^2}^2+C\|\theta_0\|_{L^2}^2t^2}{2\nu},\\
 &\int_0^t \|\nabla \theta(s,\cdot)\|_{L^2}^2ds \leq \frac{\|\theta_0\|_{L^2}^2}{2\kappa}.
\end{align*}
Then ,
\begin{align*}%\label{3.16}
 Y(t)\leq Y(0)\exp\bigg[ C_{\nu,\kappa}^\prime (1+\|u_0\|_{L^2}^2)\bigg(t+\frac{\|u_0\|_{L^2}^2}{2\nu}+\frac{C\|\theta_0\|_{L^2}^2 t^2}{2\nu}+\frac{\|\theta_0\|_{L^2}^2}{2\kappa} \bigg) \bigg].
\end{align*}
Then it gives an upper bound of $Y(t)$, which also implies for any $T>0$,
\begin{align*}
 u\in L^\infty\big([0,T]\,;H^1_\sigma(\mathbb{T}^2)\big),\quad \theta\in L^\infty\big([0,T]\,;H^1(\mathbb{T}^2)\big).
\end{align*}
If we integrate \eqref{3.12} from 0 to $T$ for some fixed $T>0$, we obtain
\begin{align*}
 u\in L^2\big([0,T]\,;H^2_\sigma(\mathbb{T}^2)\big),\quad \theta\in L^2\big([0,T]\,;H^2(\mathbb{T}^2)\big).
\end{align*}
With these facts, one can obtain
\begin{equation*}
 \frac{\partial u}{\partial t}\in L^2\big([0,T]\,;L^2_\sigma (\mathbb{T}^2)\big),\quad \frac{\partial \theta}{\partial t}\in L^2\big([0,T]\,;L^2 (\mathbb{T}^2)\big)
\end{equation*}
Combining these estimates, one can also show the solution $(u,\theta)$ is actually satisfying
\begin{equation*}
  u \in C\big([0,T]\,; H^1_\sigma (\mathbb{T}^2)\big),\ \theta \in C\big([0,T]\,; H^1 (\mathbb{T}^2)\big),\ \ \forall\ T>0.
\end{equation*}
The continuity of $t$ in $H^1_\sigma(\mathbb{T}^2)$ for $u$ and in $H^1(\mathbb{T}^2)$ for $\theta$ follows the arguments of Temam [Chap 3. \cite{RT}].
This proves the global strong solution in two-dimensional space.

\noindent
{\bf Case of  $N=3$.} We apply the Young's inequality on \eqref{3.9} and \eqref{3.11} to bound the right hand side of \eqref{3.6} and \eqref{3.7},
\begin{align*}
 \frac{1}{2}\frac{d}{dt}&\big(\|\Lambda u\|_{L^2(\mathbb{T}^3)}^2 +\|\Lambda\theta\|_{L^2(\mathbb{T}^3)}^2 \big) +\nu\|\Lambda^2 u \|_{L^2(\mathbb{T}^3)}^2+\kappa\|\Lambda^2\theta\|_{L^2(\mathbb{T}^3)}^2\notag\\
  &\leq C\|\Lambda u\|_{L^2(\mathbb{T}^3)}^{3/2}\|\Lambda^2u\|_{L^2(\mathbb{T}^3)}^{3/2}+\|\Lambda(\theta e_n)\|_{L^2(\mathbb{T}^3)}\|\Lambda u\|_{L^2(\mathbb{T}^3)}\notag\\
  &+C\|\Lambda u\|_{L^2(\mathbb{T}^3)}^{1/2}\|\Lambda^2u\|_{L^2(\mathbb{T}^3)}^{1/2}\|\Lambda \theta\|_{L^2(\mathbb{T}^3)}\|\Lambda^2\theta\|_{L^2(\mathbb{T}^3)},
\end{align*}
thus
\begin{align}\label{3.17}
 \frac{1}{2}\frac{d}{dt}&\big(\|\Lambda u\|_{L^2(\mathbb{T}^3)}^2 +\|\Lambda\theta\|_{L^2(\mathbb{T}^3)}^2 \big) +\nu\|\Lambda^2 u \|_{L^2(\mathbb{T}^3)}^2+\kappa\|\Lambda^2\theta\|_{L^2(\mathbb{T}^3)}^2\notag\\
  &\leq \frac{\nu}{2}\|\Lambda^2u\|_{L^2(\mathbb{T}^3)}^2 +\frac{\kappa}{4}\|\Lambda^2\theta\|_{L^2(\mathbb{T}^3)}^2 +C_\nu  \|\Lambda u\|_{L^2(\mathbb{T}^3)}^6 \notag\\
  &\quad+C_{\kappa,\nu} \|\Lambda u\|^2_{L^2(\mathbb{T}^3)}\|\Lambda\theta\|_{L^2(\mathbb{T}^3)}^4+\|\Lambda \theta\|_{L^2(\mathbb{T}^3)}\|\Lambda u\|_{L^2(\mathbb{T}^3)}.
\end{align}
If we set $Z(t)=1+\|\Lambda u(t,\cdot)\|_{L^2(\mathbb{T}^3)}^2+\|\Lambda \theta(t,\cdot)\|_{L^2(\mathbb{T}^3)}^2$, we obtain
\begin{align*}
 \frac{d}{dt}Z(t)\leq C_{\nu,\kappa}^\prime Z(t)^3.
\end{align*}
This implies
\begin{align*}
 Z(t) \leq \frac{Z(0)}{\sqrt{1-2C_{\nu,\kappa}^\prime Z(0)^2t}}, \quad 0<t<\frac{1}{2C_{\nu,\kappa}^\prime Z(0)^2}.
\end{align*}
And thus
\begin{align}\label{3.18}
 &1+\|\Lambda u(t,\cdot)\|_{L^2(\mathbb{T}^3)}^2+\|\Lambda \theta(t,\cdot)\|_{L^2(\mathbb{T}^3)}^2\notag\\
& \qquad \leq 2(1+\|\Lambda u_0\|_{L^2(\mathbb{T}^3)}^2+\|\Lambda\theta_0\|_{L^2(\mathbb{T}^3)}^2),
\end{align}
if $t\leq T_1(\|\Lambda u_0\|_{L^2(\mathbb{T}^3)},\|\Lambda \theta_0\|_{L^2(\mathbb{T}^3)})=\frac{3}{8C_{\nu,\kappa}^\prime Z(0)^2}$. Then we obtain
\begin{align*}
 u\in L^\infty\big([0,T_1]\,;H^1_\sigma(\mathbb{T}^3)\big),\quad \theta\in L^\infty\big([0,T_1]\,;H^1(\mathbb{T}^3)\big)
\end{align*}
Integrating \eqref{3.17} from 0 to $T_1$, we can obtain
\begin{align*}
 u\in L^2\big([0,T_1]\,;H^2_\sigma(\mathbb{T}^3)\big),\quad \theta\in L^2\big([0,T_1]\,;H^2(\mathbb{T}^3)\big).
\end{align*}
With these facts, one can obtain
\begin{equation*}
 \frac{\partial u}{\partial t}\in L^2\big([0,T_1]\,;L^2_\sigma (\mathbb{T}^3)\big),\ \frac{\partial \theta}{\partial t}\in L^2\big([0,T_1]\,;L^2 (\mathbb{T}^3)\big).
\end{equation*}
Combining these estimates, one can then show the solution $(u,\theta)$ is actually satisfying
\begin{equation*}
 u \in C\big([0,T_1]\,; H^1_\sigma (\mathbb{T}^3)\big),\ \theta \in C\big([0,T_1]\,; H^1 (\mathbb{T}^3)\big).
\end{equation*}
This proved the local existence of strong solution for three-dimensional space.

The proof of the existence for both two dimensions and three dimensions can be made rigorous by considering the Galerkin approximation procedure, which is standard in the book \cite{RT}.
\end{proof}

\begin{remark}
The difference in the Sobolev inequality \eqref{GN-1} and \eqref{GN-2} caused by the spatial dimension $N$ makes huge influence on the lifespan of the solution, which is showed in the above proof.
\end{remark}

\section{Proof of Theorem \ref{Theorem 2.1}}\label{Section 4}

In this Section, we give the proof of Theorem \ref{Theorem 2.1}.
In order to prove the main Theorem \ref{Theorem 2.1}, we recall the following Lemmas in \cite{FT} concerning the estimates of nonlinear terms.
\begin{lemma}\label{Lemma 4.1}
Let $u,v,w$ be given in $\mathcal{D}(e^{\tau \Lambda^{1/s}}\Lambda^2)$ for $\tau>0, s>0$. Then the following estimates hold, for $N=2$ or 3,
\begin{align*}%\label{4.1}
 \big|\big(e^{\tau \Lambda^{1/s}}\Lambda (u\cdot\nabla v), e^{\tau \Lambda^{1/s}}\Lambda w\big)_{L^2(\mathbb{T}^N)}\big| \leq & C \|e^{\tau \Lambda^{1/s}}u\|_{1}^{1/2} \|e^{\tau \Lambda^{1/s}} u\|_{2}^{1/2}\notag\\
  &\quad\times \|e^{\tau \Lambda^{1/s}}v\|_{1} \|e^{\tau \Lambda^{1/s}} w\|_{2},
\end{align*}
and
\begin{align*}%\label{4.2}
 \|e^{\tau\Lambda^{1/s}}(u\cdot\nabla v) \|_{L^2(\mathbb{T}^N)} \leq C\|e^{\tau\Lambda^{1/s}}u\|_1^{1/2}\|e^{\tau\Lambda^{1/s}}u\|_2^{1/2}\|e^{\tau\Lambda^{1/s}}v\|_1,
\end{align*}
where $C>0$ is a constant.
\end{lemma}
\begin{remark}
In Lemma \ref{Lemma 4.1},  $u$ is $\mathbb{R}^N$-valued vector function while $v,w$ can be either $\mathbb{R}^N$-valued vector function or $\mathbb{R}$-valued scalar function.
\end{remark}
Then we are able to prove the Theorem \ref{Theorem 2.1} following the ideal of Foias and Temam \cite{FT}.
\begin{proof}[Proof of Theorem \ref{Theorem 2.1}]
For the sake of simplicity, we shall only consider the time variable in the real case and we remark that the results can be extended to the complex case following the same arguments of \cite{FT}. For this part, we set the radius of Gevrey class $\tau(t)=t$.

We recall that the usual strategy to approximate the Navier Stokes equation is to project the velocity field onto the divergence-free field. Here we also need to write the velocity equation of the Boussinesq system \eqref{1.1} into the following form
\begin{equation}\label{4.3}
 \frac{\partial u}{\partial t}-\nu \mathcal{P}\Delta u+\mathcal{P}(u\cdot\nabla u)=\mathcal{P}(\theta e_n),
\end{equation}
where $\mathcal{P}$ is the Helmholtz-Leray orthogonal projection, i.e., projection onto divergence-free vector fields. In this way, we eliminate the pressure term as Foias and Temam did in \cite{FT}.
Denote the operator $A=-\mathcal{P}\Delta$,  the eigenvectors $\{{\bf E}_j\}_{j=1}^\infty$ of $A$ constitutes the othonormal basis of $L^2_\sigma(\mathbb{T}^N)$. We denote the corresponding eigenvalues by $\{\lambda_j\}_{j=1}^\infty$ with $0<\lambda_1<\ldots\leq \lambda_j\leq \lambda_{j+1}\leq \ldots$, then $A{\bf E}_j=\lambda_j{\bf E}_j$. We note that $\mathcal{P}$ commutes with $-\Delta$ on the torus and $A=-\Delta$ when acting on the divergence-free vector field.

The thermal equation of the Boussinesq system \eqref{1.1} is the second order parabolic equation and the eigenvectors of $-\Delta$ also constitute the orthonormal basis of $L^2({\mathbb{T}^N})$. Let $\{{\bf e}_\alpha\}_{\alpha=1}^\infty$ be the orthonormal basis of $L^2({\mathbb{T}^N})$ and $\{\tau_\alpha\}_{\alpha=1}^\infty$ be the corresponding eigenvalues, which satisfies
$$
 0<\tau_1<\ldots\leq \tau_\alpha\leq \tau_{\alpha+1}\leq\ldots.
$$

We will then approximate the equation \eqref{4.3} and the thermal equation of \eqref{1.1} by the Galerkin method. For fixed integer $m\in\mathbb{N}$, let $\mathbb{P}_m$ be the projection from $L^2_\sigma(\mathbb{T}^N)$ onto the subspace $W_m$ spanned by $\{{\bf E}_j,\ 1\leq j\leq m, j\in\mathbb{Z}\}$ and $\mathbf{P}_m$ be the projection from $L^2(\mathbb{T}^N)$ onto the subspace $V_m$ of $L^2(\mathbb{T}^N)$ spanned by $\{{\bf e}_\alpha,\ 1\leq\alpha\leq m,\ \alpha\in\mathbb{Z} \}.$

We are looking for solution $(u_m,\theta_m)$ to the following approximate equations
\begin{align}
 &\frac{\partial u_m}{\partial t}-\nu\mathcal{P}\Delta u_m +\mathbb{P}_m\mathcal{P}(u_m\cdot\nabla u_m) =\mathbb{P}_m\mathcal{P}(\theta_m e_N),\label{4.4}\\
 &\frac{\partial \theta_m}{\partial t}-\kappa \Delta\theta_m +\mathbf{P}_m(u_m\cdot\nabla \theta_m)=0,\label{4.5}\\
 & u_m(0)=\mathbb{P}_m u_0,\ \theta(0)=\mathbf{P}_m \theta_0,\label{4.6}
\end{align}
where
$$u_m(t,x)=\sum_{1\leq j\leq m}\xi_{j,m}(t){\bf E}_j,\ \theta_\alpha(t,x)=\sum_{1\leq\alpha\leq m}\eta_{\alpha,m}(t){\bf e}_\alpha.
$$

The function $u_m, \theta_m$ are analytics for $x$ variables, in fact the sequence of ${\bf E}_j$'s and ${\bf e}_\alpha$'s is the linear combinations of the sequence of functions ${\bf W}_{k,\ell}$ and ${\bf w}_{n}$,
$$
 {\bf W}_{k,\ell}=a_{k,\ell}\big(e_\ell-\frac{k_\ell k}{|k|^2}\big)e^{i k\cdot x},\quad {\bf w}_{n}=\frac{1}{(2\pi)^N} e^{i n\cdot x},\quad \ell=1, \cdots, N,
$$
where $k=(k_1,\ldots,k_N),\, n=(n_1,\ldots,n_N)\in\mathbb{Z}^N\setminus \{0\}$,  $e_1,\ldots,e_N$ is the canonical basis of $\mathbb{R}^N$.  $a_{k,\ell}$ is the coefficient that make the sequence ${\bf W}_{k,\ell}$ to be orthonormal in $L^2_\sigma(\mathbb{T}^N)$. 

 After taking $L^2$-inner product of \eqref{4.4} with ${\bf E}_j$ and taking $L^2$-inner product of \eqref{4.5} with ${\bf e}_\alpha$, for $1\leq j\leq m$ and $1\leq\alpha\leq m$. The Cauchy problem \eqref{4.4}-\eqref{4.6} is equivalent to the following Cauchy problem for ordinary differential system
\begin{align}
 &\frac{d}{dt}\xi_{j,m}(t)+\nu \lambda_j \xi_{j,m}(t)+\sum_{1\leq k,\ell\leq m}A_{k,\ell,j}\xi_{k,m}(t)\xi_{\ell,m}(t)\notag\\
 &\qquad\qquad\qquad\qquad\qquad\qquad=\sum_{1\leq\gamma\leq m}C_{\gamma,j}\eta_{\gamma,m}(t),\label{4.7}\\
 &\frac{d}{dt}\eta_{\alpha,m}(t)+\kappa \tau_\alpha \eta_{\alpha,m}(t)+\sum_{1\leq j,\beta\leq m}B_{j,\beta,\alpha}\xi_{j,m}(t)\eta_{\beta,m}(t)=0,\label{4.8}\\
 & \xi_{j,m}(0)=(u_0,{\bf E}_j),\quad  \eta_{\alpha,m}(0)=(\theta_0, {\bf e}_\alpha),\label{4.9}
\end{align}
where 
\begin{align*}
A_{k,\ell,j}&=({\bf E}_k\cdot\nabla {\bf E}_\ell,\, {\bf E}_j)_{L^2(\mathbb{T}^N)}, \\
B_{j,\beta,\alpha}&=({\bf E}_j\cdot\nabla  {\bf e}_\beta,\, {\bf e}_\alpha)_{L^2(\mathbb{T}^N)},\\
C_{\gamma,j}&=({\bf e}_\gamma e_N, \, {\bf E}_j)_{L^2(\mathbb{T}^N)}.
\end{align*}
The standard theory of ordinary differential equations indicates that the ODE system \eqref{4.7}-\eqref{4.9} admet a unique local  solution $\big(\xi_{j,m}(t),\eta_{\alpha,m}(t)\big)_{1\leq j,\alpha\leq m}$ on some interval $[0,T_m]$. 

We prove now the solution of the ODE system \eqref{4.7}-\eqref{4.9},  $\big(\xi_{j,m}(t),\eta_{\alpha,m}(t)\big)$ can be extended to a global in time solution for any fixed $m$. To show this, we note that each ${\bf E}_j$ and ${\bf e}_\alpha$ are analytics. Then we can perform integration by parts to obtain
\begin{align}
A_{k,\ell,j} &=({\bf E}_k\cdot\nabla{\bf E}_\ell,\, {\bf E}_j)_{L^2(\mathbb{T}^N)}\notag\\
&=-({\bf E}_k\cdot\nabla {\bf E}_j,\, {\bf E}_\ell)_{L^2(\mathbb{T}^N)}\notag\\
&=-A_{k,j,\ell},\quad  \forall\ 1\leq k,\ell,j\leq m,\label{4.10}
\end{align}
where we used the fact $\nabla\cdot {\bf E}_k=0$.
Integrating by parts, one can also obtain
\begin{equation}\label{4.11}
B_{j,\beta,\alpha}=-B_{j,\alpha,\beta},\quad  \forall\ 1\leq j,\beta,\alpha\leq m.
\end{equation}
So if we multiply \eqref{4.7} by $\xi_{j,m}(t)$ and take sum over $\{1\leq j\leq m,j\in\mathbb{Z}\}$, we obtain
\begin{equation}\label{4.12}
 \frac{1}{2}\frac{d}{dt}\sum_{j=1}^m\xi_{j,m}(t)^2+\nu\sum_{j=1}^m \lambda_j \xi_{j,m}(t)^2=\sum_{j=1}^m\sum_{\gamma=1}^mC_{\gamma,j}\eta_{\gamma,m}(t)\xi_{j,m}(t),
\end{equation}
where we infer from \eqref{4.10} the following fact
$$
 \sum_{1\leq k,\ell,j\leq m}A_{k,\ell,j}\, \xi_{k,m}\, \xi_{\ell,m}\, \xi_{j,m}= -\sum_{1\leq k,\ell,j\leq m}A_{k,j,\ell}\, \xi_{k,m}\, \xi_{j,m}\, \xi_{\ell,m}=0.
$$
If we multiply \eqref{4.8} by $\eta_{\alpha,m}(t)$ and take sum over $\{1\leq\alpha\leq m,\alpha\in\mathbb{Z}\}$, we obtain
\begin{equation}\label{4.13}
 \frac{1}{2}\frac{d}{dt}\sum_{\alpha=1}^m\eta_{\alpha,m}(t)^2+\kappa\, \tau_\alpha \sum_{\alpha=1}^m \eta_{\alpha,m}^2=0,
\end{equation}
where we infer from \eqref{4.11} the fact
$$
 \sum_{1\leq j,\alpha,\beta\leq m}B_{j,\alpha,\beta}\, \xi_{j,m}\, \eta_{\alpha,m}\, \eta_{\beta,m}=-\sum_{1\leq j,\alpha,\beta\leq m}B_{j,\beta,\alpha}\, \xi_{j,m}\, \eta_{\beta,m}\, \eta_{\alpha,m}=0.
$$
The equation \eqref{4.13} implies that, for any $ m\in\mathbb{N},$
\begin{equation*}%\label{4.14}
 \|\theta_m(t,\cdot)\|_{L^2(\mathbb{T}^N)}^2=\sum_{\alpha=1}^m\eta_{\alpha,m}(t)^2\leq \sum_{\alpha=1}^m\eta_{\alpha,m}(0)^2  \leq \|\theta_0\|_{L^2(\mathbb{T}^N)}^2,\, \forall\, t>0.
\end{equation*}
With this fact, one can infer from \eqref{4.12} that
\begin{align}\label{4.15}
 \frac{d}{dt}\|u_m(t,\cdot)\|_{L^2(\mathbb{T}^N)}
 \leq \|\theta_m(t,\cdot)\|_{L^2(\mathbb{T}^N)},
\end{align}
by noting that $\|u_m(t,\cdot)\|_{L^2(\mathbb{T}^N)}^2=\sum_{j=1}^m \xi_{j,m}(t)^2$.
If we integrate \eqref{4.15} from 0 to $t$, we obtain
\begin{align*}
 \|u_m(t,\cdot)\|_{L^2(\mathbb{T}^N)} &\leq \|u_m(0,\cdot)\|_{L^2(\mathbb{T}^N)}+\int_0^t \|\theta_m(s,\cdot)\|_{L^2(\mathbb{T}^N)}ds,\notag\\
  &\leq \|u_0\|_{L^2(\mathbb{T}^N)}^2+t \|\theta_0\|_{L^2(\mathbb{T}^N)},\quad \forall\, m\in \mathbb{N},\ \forall\, t>0.
\end{align*}
The above priori estimates show that the solution $(\xi_{j,m(t)},\eta_{\alpha,m}(t))$ is bounded only by the initial data, then we can repeat the arguments above to extend the local solution to arbitrary time interval $[0,T]$. 
Thus for fixed integer $m$, the Cauchy problem of ordinary differential equations \eqref{4.7}-\eqref{4.9} possess a unique solution $(\xi_{j,m}(t), \eta_{\alpha,m}(t))\in C^1([0,T])$ for all $T>0$. Equivalently, we have the solution $(u_m,\theta_m)$ to the approximate equation \eqref{4.4}-\eqref{4.6} satisfies
$$
 u_m\in C^1\big([0,T]\,;H^r_\sigma(\mathbb{T}^N) \big),\ \theta_m\in C^1\big([0,T]\,;H^r(\mathbb{T}^N) \big),\ \forall\, T>0,\ \forall\, r>0,
$$
for fixed $m\in\mathbb{N}$.

Once we obtained the approximating solution $\big(u_m,\theta_m \big)$, we then perform the Gevrey norm estimates on $(u_m,\theta_m)$.  For fixed integer $m$, as a finite summation of ${\bf E}_j$ and ${\bf e}_\alpha$, the solution $\big(u_m,\theta_m\big)$ belongs to Gevrey class $\mathcal{D}(\Lambda^r e^{\tau\Lambda})$ for arbitrary $r>0$. 
Let $T>0$ be fiexed. For $0<t<T$, we take $L^2$-inner product of the velocity equation of \eqref{4.4} with $\Lambda^2 e^{2t \Lambda}u_m$, which gives
\begin{align}\label{4.16}
  &\big(\Lambda e^{t \Lambda}\frac{d}{dt} u_m(t),\Lambda e^{t \Lambda}u_m(t)\big)_{L^2(\mathbb{T}^N)}+\nu \|\Lambda^2 e^{t \Lambda} u_m(t)\|_{L^2(\mathbb{T}^N)}^2 \notag\\
  &=\big(\Lambda e^{t \Lambda}(\theta_m e_N),e^{t \Lambda} \Lambda u_m(t)\big)_{L^2(\mathbb{T}^N)}\\
  &\qquad\qquad-\big(\Lambda e^{t \Lambda}(u_m\cdot\nabla u_m),e^{t \Lambda} \Lambda u_m\big)_{L^2(\mathbb{T}^N)},\notag 
\end{align}
where we used the fact that $u_m$ is divergence free and $\mathcal{P}$ is symmetric.
Taking the $L^2$-inner product of the thermal equation of \eqref{4.5} with $\Lambda^2 e^{2t \Lambda}\theta_m$, we obtain, similarly,
\begin{align}\label{4.17}
 \big(\Lambda e^{t \Lambda}\frac{d}{dt} \theta_m(t) , \Lambda e^{t \Lambda}\theta_m \big)_{L^2(\mathbb{T}^N)}+\kappa \|\Lambda^2 e^{t \Lambda}\theta_m \|_{L^2(\mathbb{T}^N)}^2\notag \\
 +\big(\Lambda e^{t \Lambda}(u_m\cdot\nabla\theta_m) ,\Lambda e^{t \Lambda}\theta_m \big)_{L^2(\mathbb{T}^N)} =0.
\end{align}
By Plancherel's theorem and summing over \eqref{4.16} and \eqref{4.17}, we can write
\begin{align}\label{4.18}
 \frac{1}{2}\frac{d}{dt} &\big(\|\Lambda e^{t \Lambda} u_m(t) \|_{L^2(\mathbb{T}^N)}^2 +\|\Lambda e^{t \Lambda} \theta_m(t) \|_{L^2(\mathbb{T}^N)}^2\big)\notag\\
 &\quad+\nu \|\Lambda^2 e^{t \Lambda} u_m(t)\|_{L^2(\mathbb{T}^N)}^2 + \kappa \|\Lambda^2 e^{t \Lambda}\theta_m \|_{L^2(\mathbb{T}^N)}^2  \notag\\
 &=\big(\Lambda^{2}e^{t\Lambda}u_m,\Lambda e^{t\Lambda}u_m \big)_{L^2(\mathbb{T}^N)}+\big(\Lambda^{2}e^{t\Lambda}\theta_m ,\Lambda e^{t\Lambda}\theta_m \big)_{L^2(\mathbb{T}^N)} \notag\\
 &\quad +\big(\Lambda e^{t \Lambda}(\theta_m e_N),e^{t \Lambda} \Lambda u_m(t)\big)_{L^2(\mathbb{T}^N)}\\
 &\qquad\qquad-\big(\Lambda e^{t \Lambda}(u_m\cdot\nabla u_m),e^{t \Lambda} \Lambda u_m\big)_{L^2(\mathbb{T}^N)} \notag\\
 &\quad- \big(\Lambda e^{t \Lambda}(u_m\cdot\nabla\theta_m) ,\Lambda e^{t \Lambda}\theta_m \big)_{L^2(\mathbb{T}^N)}.\notag 
\end{align}
By Cauchy-Schwartz inequality, we have
\begin{align*}%\label{4.19}
 \big|\big(\Lambda^{2}e^{t\Lambda}u_m,\Lambda e^{t\Lambda}u_m \big)_{L^2(\mathbb{T}^N)}\big| \leq \frac{\nu}{8}\|\Lambda^{2}e^{t\Lambda}u_m\|_{L^2(\mathbb{T}^N)}^2 +C_\nu \|\Lambda e^{t\Lambda}u_m\|_{L^2(\mathbb{T}^N)}^2,
\end{align*}
and
\begin{align*}%\label{4.20}
 \big|\big(\Lambda^{2}e^{t\Lambda}\theta_m,\Lambda e^{t\Lambda}\theta_m \big)_{L^2(\mathbb{T}^N)}\big| \leq \frac{\kappa}{8}\|\Lambda^{2}e^{t\Lambda}\theta_m\|_{L^2(\mathbb{T}^N)}^2 +C_\kappa \|\Lambda e^{t\Lambda}\theta_m\|_{L^2(\mathbb{T}^N)}^2.
\end{align*}
From Lemma \ref{Lemma 4.1}, we have,
\begin{align}\label{4.21}
 \big|\big(\Lambda e^{t\Lambda}(u_m\cdot\nabla u_m) ,\Lambda e^{t\Lambda}u_m \big)_{L^2(\mathbb{T}^N)}\big|
 \leq C \|e^{t\Lambda}u_m\|_{1}^{3/2} \|e^{t\Lambda}u_m\|_{2}^{3/2},
\end{align}
and
\begin{align}
 \big|\big(\Lambda e^{t\Lambda}(u_m\cdot\nabla \theta_m) &,\Lambda e^{t\Lambda}\theta_m \big)_{L^2(\mathbb{T}^N)}\big| \notag\\
 &\leq C \|e^{t\Lambda}u_m\|_{1}^{1/2} \|e^{t\Lambda}u_m\|_{2}^{1/2}\|e^{t\Lambda} \theta_m\|_{1}\|e^{t\Lambda} \theta_m\|_{2}.\label{4.22}
\end{align}
Noting that on the torus $\mathbb{T}^N$ we have the following Poincar\'e inequality,
\begin{align*}
 \|e^{t\Lambda}u_m \|_{L^2(\mathbb{T}^N)} \leq C \|e^{t\Lambda}u_m \|_{1}, \quad \|e^{t\Lambda}\theta_m \|_{L^2(\mathbb{T}^N)} \leq C \|e^{t\Lambda}\theta_m \|_{1},
\end{align*}
for some constant $C$ independent of $m$. Using Young's inequality, the right hand side of \eqref{4.21} can be bounded by
\begin{align}\label{4.23}
 \big|\big(\Lambda e^{t\Lambda}(u_m\cdot\nabla u_m) ,\Lambda e^{t\Lambda}u_m \big)_{L^2(\mathbb{T}^N)}\big| \leq \frac{\nu}{8}\|e^{t\Lambda}u_m\|_2^2+C_{\nu}^\prime\|e^{t\Lambda}u_m \|_{1}^6.
\end{align}
By applying the Cauchy-Schwartz inequality twice, the right hand side of \eqref{4.22} can be bounded by
\begin{align}\label{4.24}
 \big|\big(\Lambda e^{t\Lambda} &(u_m\cdot\nabla \theta_m) ,\Lambda e^{t\Lambda}\theta_m \big)_{L^2(\mathbb{T}^N)}\big| \notag\\
 &\leq C_\kappa^\prime \|e^{t\Lambda}u_m\|_1 \|e^{t\Lambda}u_m\|_2 \|e^{t\Lambda}\theta_m\|_1^2 +\frac{\kappa}{8}\|e^{t\Lambda}\theta_m\|_2^2 \notag\\
 &\leq C_{\kappa,\nu}\|e^{t\Lambda}u_m\|_1^2 \|e^{t\Lambda}\theta_m\|_1^4+\frac{\nu}{8}\|e^{t\Lambda}u_m\|_2^2+\frac{\kappa}{8}\|e^{t\Lambda}\theta_m\|_2^2,
\end{align}
where \eqref{4.23} and \eqref{4.24} are both valid for the two dimensions $N=2$ and three dimensions $N=3$.
Then the right hand side of \eqref{4.18} is bounded by
\begin{align}\label{4.25}
 \frac{1}{2}\frac{d}{d t} &\big(\|\Lambda e^{t\Lambda}u_m(t,\cdot)\|_{L^2(\mathbb{T}^N)}^2+\|\Lambda e^{t\Lambda}\theta_m(t,\cdot)\|_{L^2(\mathbb{T}^N)}^2\big)+\frac{3\nu}{4}\|\Lambda^2 e^{t\Lambda}u_m\|_{L^2(\mathbb{T}^N)}^2 \notag\\
 &+\frac{3\kappa}{4}\|\Lambda^2 e^{t\Lambda}\theta_m\|_{L^2(\mathbb{T}^N)}^2\leq C_\nu\|\Lambda e^{t\Lambda}\|_{L^2(\mathbb{T}^N)}^2+C_\kappa \|\Lambda e^{t\Lambda}\theta_m\|_{L^2(\mathbb{T}^N)}^2 \notag\\
 &\quad+C_\nu^\prime \|\Lambda e^{t\Lambda} u_m\|_{L^2(\mathbb{T}^N)}^6 +C_{\kappa,\nu}\|\Lambda e^{t\Lambda}u_m\|_{L^2(\mathbb{T}^N)}^2 \|\Lambda e^{t\Lambda}\theta_m\|_{L^2(\mathbb{T}^N)}^4
\end{align}
Now if we set $X_m(t)=1+\|\Lambda e^{t\Lambda} u_m(t,\cdot)\|_{L^2(\mathbb{T}^N)}^2+\|\Lambda e^{t\Lambda} \theta_m(t,\cdot)\|_{L^2(\mathbb{T}^N)}^2$, then we have
\begin{align*}
 \frac{d}{dt}X_m(t)\leq C_{\nu,\kappa}^\prime X_m(t)^3,
\end{align*}
where $C_{\nu,\kappa}^\prime$ is a constant large enough depending on $\nu,\kappa$. We then obtain, for small $t$,
\begin{align*}
 X_m(t)\leq \frac{X_m(0)}{\sqrt{1-2C_{\nu,\kappa}^\prime X_m(0)^2 t}}\, .
\end{align*}
This gives
\begin{align*}
 X_m(t) &=1+\|\Lambda e^{t \Lambda}u_m(t,\cdot)\|_{L^2(\mathbb{T}^N)}^2+\|\Lambda e^{t \Lambda} \theta_m(t,\cdot)\|_{L^2(\mathbb{T}^N)}^2\\
       &\leq 2X_m(0)=2+2\|\Lambda u_m(0)\|_{L^2(\mathbb{T}^N)}^2+2\|\Lambda \theta_m(0)\|_{L^2(\mathbb{T}^N)}^2\\
       &\leq 2+2\|u_0\|_{L^2(\mathbb{T}^N)}^2+\|\theta_0\|_{L^2(\mathbb{T}^N)}^2\, ,
\end{align*}
for $0\le t\le T_1$ with
\begin{equation*}
T_1\big(\|\Lambda u_0\|_{L^2(\mathbb{T}^N)},\|\Lambda \theta_0\|_{L^2(\mathbb{T}^N)}\big)=\frac{3}{8C_{\nu,\kappa}^\prime}\big(1+\|\Lambda u_0\|_{L^2(\mathbb{T}^N)}^2+\|\Lambda \theta_0\|_{L^2(\mathbb{T}^N)}^2 \big)^{-2}.
\end{equation*}
Then $(u_m(t),\theta_m(t))\in \big(\mathcal{D}(\Lambda e^{t \Lambda}),\mathcal{D}(\Lambda e^{t \Lambda})\big)$ for $0<t<T_1$. In particular, for $0<t<T_1$, we have
\begin{align}\label{4.26}
 \|\Lambda e^{t \Lambda} u_m(t,\cdot)\|_{L^2(\mathbb{T}^N)}^2 &+\|\Lambda e^{t \Lambda} \theta_m(t,\cdot)\|_{L^2(\mathbb{T}^N)}^2\notag\\
 &\leq 2+2\|\Lambda u_0\|_{L^2(\mathbb{T}^N)}^2+2\|\Lambda \theta_0\|_{L^2(\mathbb{T}^N)}^2,\ \forall\, m,
\end{align}
That is
\begin{equation*}
 e^{t\Lambda}u_m,\,\, e^{t\Lambda}\theta_m\in L^\infty\big([0,T_1]\,; H^1\big),\ \forall\, m.
\end{equation*}
Moreover, integrating \eqref{4.25} from 0 to $T_1$ one obtain
\begin{equation}\label{4.27}
 e^{t\Lambda}u_m, \, e^{t\Lambda}\theta_m\in L^2\big([0,T_1]\,; H^2\big),\ \forall\, m.
\end{equation}
In order to use the compactness theorem of \cite{RT}, we also need to obtain the estimates for $\frac{d}{dt}\big(e^{t\Lambda}u_m(t) \big)$ and $\frac{d}{dt}\big(e^{t\Lambda}\theta_m(t)\big)$. To do this, we return to equation \eqref{4.4} and \eqref{4.5}. We apply $e^{t\Lambda}$ on both sides of \eqref{4.4},
\begin{align}\label{4.28}
 \frac{\partial}{\partial t}\big(e^{t\Lambda}u_m(t)\big)-\Lambda e^{t\Lambda}u_m
 &-\nu e^{t\Lambda}\Delta u_m +e^{t\Lambda}\mathbb{P}_m\mathcal{P}(u_m\cdot\nabla u_m)\notag\\
 &=e^{t\Lambda}\mathbb{P}_m\mathcal{P}(\theta_m e_n),
\end{align}
In the identity \eqref{4.28}, we note
\begin{equation*}
 \Lambda e^{t\Lambda} u_m \in L^\infty\big([0,T_1]\,;L^2(\mathbb{T}^N)\big),\ e^{t\Lambda}\mathbb{P}_m\mathcal{P}(\theta_m e_n)\in L^\infty\big([0,T_1]\,;L^2(\mathbb{T}^N)\big),\ \forall\, m.
\end{equation*}
By \eqref{4.27}, we have
\begin{equation*}
 e^{t\Lambda}\Delta u_m \in L^2\big([0,T_1]\,;L^2(\mathbb{T}^N)\big),\ \forall\, m.
\end{equation*}
Using Lemma \ref{Lemma 4.1}, we have
\begin{equation*}
 \|e^{t\Lambda}\mathcal{P}(u_m\cdot\nabla u_m)\|_{L^2(\mathbb{T}^N)} \leq C\|e^{t\Lambda}u_m\|_1^{3/2}\|e^{t\Lambda}u_m\|_2^{1/2}.
\end{equation*}
Then we have
\begin{equation*}
 e^{t\Lambda}\mathbb{P}_m\mathcal{P}(u_m\cdot\nabla u_m)\in L^4\big([0,T_1]\,; L^2(\mathbb{T}^N)\big).
\end{equation*}
Thus we infer from \eqref{4.28} and the above facts,
\begin{equation}\label{4.29}
 \frac{\partial}{\partial t}\big(e^{t\Lambda}u_m(t) \big)\in L^2\big([0,T_1]\,;L^2(\mathbb{T}^N) \big).
\end{equation}
By applying $e^{t\Lambda}$ on both sides of \eqref{4.5},  we have
\begin{equation*}%\label{4.30}
 \frac{\partial}{\partial t}\big(e^{t\Lambda}\theta_m(t) \big)-\Lambda e^{t\Lambda}\theta_m -\kappa e^{t\Lambda}\Delta \theta_m +e^{t\Lambda}\mathbf{P}_m(u_m\cdot\nabla \theta_m)=0.
\end{equation*}
From \eqref{4.26}, we have
\begin{equation*}
 \Lambda e^{t\Lambda}\theta_m \in L^\infty\big([0,T_1]\,;L^2(\mathbb{T}^N)\big)
\end{equation*}
By \eqref{4.27}, we have
\begin{equation*}
 \kappa e^{t\Lambda}\Delta \theta_m \in L^2\big([0,T_1]\,;L^2(\mathbb{T}^N)\big)
\end{equation*}
Using Lemma \ref{Lemma 4.1}, we have
\begin{equation*}
 \|e^{t\Lambda}\mathbf{P}_m(u_m\cdot\nabla\theta_m)\|_{L^2(\mathbb{T}^N)}\leq C\|e^{t\Lambda}u_m\|_1^{1/2}\|e^{t\Lambda}u_m\|_2^{1/2}\|e^{t\Lambda}\theta_m\|_1.
\end{equation*}
Therefore
\begin{equation*}
 e^{t\Lambda}\mathbf{P}_m(u_m\cdot\nabla\theta_m)\in L^4\big([0,T_1]\,;L^2(\mathbb{T}^N) \big)
\end{equation*}
Thus we have
\begin{equation}\label{4.31}
 \frac{\partial}{\partial t}\big(e^{t\Lambda}\theta_m(t) \big)\in L^2\big([0,T_1]\,;L^2(\mathbb{T}^N) \big)\, .
\end{equation}
With \eqref{4.26},\eqref{4.27},\eqref{4.29} and \eqref{4.31}, we obtain the limit solution $\big(E(u),F(\theta)\big)$ up to some subsequence of $(e^{t\Lambda} u_m, e^{t\Lambda} \theta_m)$ such that
\begin{equation*}
  e^{t\Lambda}u_m \to E(u),\  e^{t\Lambda}\theta_m \to F(\theta) \in L^2\big([0,T_1]\,;H^1(\mathbb{T}^N) \big),
\end{equation*}
where we used the compact embedding theorem \cite{RT}. To show that $E(u)=e^{t\Lambda}u$ and $F(\theta)=e^{t\Lambda}\theta$, one only need to recall that
\begin{equation*}
 u_m \to u,\ \theta_m\to \theta,\ \in L^2\big([0,T_1]\,;H^1(\mathbb{T}^N) \big).
\end{equation*}
Then we have, by the uniqueness of the strong solution,
\begin{equation*}
 e^{t\Lambda}u_m\to e^{t\Lambda}u,\ e^{t\Lambda}\theta_m\to e^{t\Lambda}\theta, \ \in L^2\big([0,T_1]\,;H^1(\mathbb{T}^N) \big).
\end{equation*}
For now we have proved the local solution with values in Gevrey class functions.
Now if we know that
\begin{equation}\label{4.33}
 \|\Lambda u(t,\cdot)\|_{L^2(\mathbb{T}^N)}^2+\|\Lambda \theta(t,\cdot)\|_{L^2(\mathbb{T}^N)}^2\leq M_{0,T},\quad \forall\, 0\leq t\leq T,
\end{equation}
for some large positive number depending on $T$ and the initial data,
then we can repeat the argument above at any time $0<t<T_1$ and find that the solution can be extended to arbitrary $T$ in two-dimensional space. 

But the estimate \eqref{4.33} is not true for three dimensions $N=3$, because we do not have the uniform bound \eqref{3.18} for all $t$ in three dimensions. 
\end{proof}

\bigskip
\noindent{\bf Acknowledgements.} 
The research of the second author is supported partially by
``The Fundamental Research Funds for Central Universities of China".

\end{document}